\documentclass[leqno,english]{article}

\usepackage[latin1]{inputenc}
\usepackage[french,english]{babel}
\usepackage{amsfonts}
\usepackage{amsmath}
\usepackage{amssymb}
\usepackage{eufrak}
\usepackage{graphicx}
\usepackage{mathrsfs}
\usepackage{pstricks,pst-node}
\usepackage{young}

\newcommand{\ds}{\displaystyle}

\newcommand{\R}{\mathbb{R}}

\newcommand{\Np}[2]{[\![#1,#2]\!]}
\newcommand{\Q}{\mathbb{Q}}
\newcommand{\Z}{\mathbb{Z}}

\newcommand{\C}{\mathbb{C}}

\newcommand{\A}{\mathbb{A}}

\newcommand{\Sc}{\mathbb{S}}
\newcommand{\D}{\mathbb{D}}
\newcommand{\val}{\mathfrak{V}}

\newcommand{\ida}{\mathfrak{a}}
\newcommand{\idA}{\mathfrak{A}}
\newcommand{\idc}{\mathfrak{c}}

\newcommand{\ido}{\mathfrak{o}}

\newcommand{\norm}{\mathcal{N}}
\newcommand{\fhh}{\mathcal{P}_{n}(k)}

\newcommand{\be} {\begin{equation}}
\newcommand{\ee} {\end{equation}}
\newcommand{\bea} {\begin{eqnarray}}
\newcommand{\eea} {\end{eqnarray}}
\newcommand{\Bea} {\begin{eqnarray*}}
\newcommand{\Eea} {\end{eqnarray*}}
\newcommand{\pa} {\partial}

\newcommand{\al} {\alpha}
\newcommand{\ba} {\beta}
\newcommand{\de} {\delta}
\newcommand{\ga} {\gamma}

\newcommand{\om} {\omega}
\newcommand{\De} {\Delta}
\newcommand{\la} {\lambda}
\newcommand{\si} {\sigma}

\newcommand{\La} {\Lambda}

\newcommand{\na} {\nabla}
\newcommand{\va} {\varphi}
\newcommand{\eps} {\varepsilon}

\renewcommand{\le}{\leqslant}
\renewcommand{\ge}{\geqslant}

\newcommand{\I}{\mbox{I}}
\newcommand{\Id}{\mbox{Id}}

\newcommand{\GL}{\mathrm{GL}}
\newcommand{\SL}{\mathrm{SL}}

\DeclareMathOperator{\tr} {Tr}
\newcommand{\transp}[1]{#1 '}

\usepackage{theorem}
\newtheorem{theo}{Theorem}
\newtheorem{prop}[theo]{Proposition}
\newtheorem{lemm}[theo]{Lemma}
\theorembodyfont{\upshape}
\newtheorem{exem}[theo]{Example}
\newtheorem{defi}[theo]{Definition}
\newtheorem{rema}[theo]{Remark}
\newtheorem{proof}[theo]{Proof}

\author{Bertrand Meyer}
\date{}

\title{Generalised Hermite constants, \\ Voronoi theory and heights on flag varieties}

\begin{document}

\maketitle

\begin{abstract}
This paper explores the study of the general Hermite constant associated to the general linear group and its irreducible representations, as defined by T. Watanabe in \cite{wa1}. To that end, a height, which naturally applies to flag varieties, is built and notions of perfection and eutaxy characterising extremality are introduced. Finally we acquaint some relations ({\it e.g.} with Korkine--Zolotareff reduction), upper bounds and computation relative to these constants.
\end{abstract}

\section*{Introduction}

The traditional Hermite constant can be defined by the following formula
\be \label{eq:0} \ga_n =\max_{A} \min_{ \substack{x \in \Z^n, \\ \, x \neq 0} } \frac{A[x]}{(\det(A))^{1/n} }   \ee
when $A$ runs through the set of all positive definite quadratic forms, or else, from the lattice standpoint, by the equivalent formula

\be \ga_n = \max_{\La} \frac{ \min{\La} }{ (\det(\La))^{1/n}} \ee
where $\La$ stands for a lattice of $\R^n$.

These constants appear in various areas ; in particular, they account for the highest density one can reach by regularly packing balls of equal radius. 

Diverse generalisation of these constants has been set forth, the most accomplished taking the following shape  \cite{wa1} :
\be \label{eq:1} \ga_{\pi} (\| \cdot \|_{\A_k}) = \max_{g \in G(\A_k)^1} \min_{\ga \in G(k)} \| \pi(g\ga) x_{\pi} \|_{\A_k}^{2/[k:\Q]} .\ee
In this formula, an algebraic number field $k$ is fixed, as well as a connected reductive algebraic group $G$. The notation $G(\A_k)^1$ stands for the unimodular part ({\it i.e.} the intersection of the kernels of the characters of the group $G(\A_k)$).
Besides, $\pi$ is an irreducible strongly rational representation, $x_{\pi}$ is a highest weight vector of the representation, $\A_k$ is the ring of the ad\`eles on $k$ and $\| \cdot \|_{\A_k}$ denotes a height on $\Sc_{\pi}(k^n)$, the vector space that carries the action of the representation $\pi$.

When $\pi$ is the natural representation of $\GL_n$ on $k^n$ (\emph{i.e.} when for any $ x \in k^n$ and any $g \in \GL_n(k)$, $\pi(g)x$ is simply $g(x)$), we recover the Hermite--Humbert constant \cite{ica} and in particular the traditional Hermite constant expounded above (equation (\ref{eq:0})) when in addition $k$ is the rationnal field. Likewise if $\pi$ is the representation on the exterior power $\bigwedge^d (k^n)$, we get the Rankin-Thunder constant \cite{thu2}, or simply the Rankin constant \cite{ran} if in addition $k$ is the field of rationnals.

\medskip

The constant $\ga_{\pi}^{\GL_n} (\| \cdot \|_{\A_k})$ admits also a geometrical interpretation. Indeed, let us define $Q_{\pi}$ the (parabolic) subgroup of $G$ which stabilises the line spanned by the highest weight vector $x_{\pi}$. The map 
\be \label{e:pldrap} g \mapsto \pi(g^{-1}) x_{\pi} \ee
provides an embedding of the flag variety $ Q_{\pi} \backslash \GL_n$ into the projective space $\mathbb{P} \left( \Sc_{\pi}(k^n) \right)$. For $A \in \GL_n(\A_k)$, and $\mathscr{D}$ a flag represented by $x \in \Sc_{\pi}(k^n)$, one can define  the twisted height $H_A$ by $H_A( \mathscr{D} ) = \| Ax\|_{\A_k}$. Let us denote by $m$ the sum of the dimensions of the nested spaces of the flag $\mathscr{D}$. Then the generalised Hermite constant can be read into as the smallest constant $C$ such that for any $A \in \GL_n (\A_k)$, there exists a flag $\mathscr{D}$ satisfying 
$$ H_A(\mathscr{D}) \le C^{1/2} |\det (A) |_{\A_k}^{m/n}, $$
which joins up with the definition by J. L. Thunder in \cite{thu2} as far as subspaces of $k^n$ of fixed dimension are concerned.

In the case of the traditional Hermite constant, G. Voronoï stated two properties, \emph{perfection} and \emph{eutaxy}, which enable to characterise extreme quadratic forms, or in other words, forms that constitute a local maximum of the quotient $\min A / \det(A)^{1/n} $. Generalisations of the notions of eutaxy and perfection have been put forward to fit in the framework of the Rankin \cite{cou3} or Hermite-Humbert constants \cite{cou1}. The point of this paper is to define  appropriate notions in the case of the constant $\ga_{\pi}(\| \cdot \|_{\A_k})$ associated to any irreducible polynomial representation $\pi$ of the group $\GL_n $.

Our text is organised as follows. In a first part, we fix the conventions we shall stick to in the sequel ; we shall recall what is to be known about irreducible representations of $GL_n$ ; we shall also give a detailed construction of the height that is let invariant by the action of the compact subgroup $K_n (\A_k) = \displaystyle \prod_{v \in \val_{\infty}} O_n(k_n) \, \times \, \prod_{v \in \val_f} \GL_n(\ido_v)$ (Think of this subgroup as an adelic analog of the orthogonal group in the real case). In a second part, we shall commit ourselves to exhibit a link between the adelic definition of $\ga_{\pi}(\| \cdot \|_{\A_k})$ with an {\it ad hoc} definition built on Hermite--Humbert forms. This second definition has the advantage of relying only on finitely many places of $k$ : the archimedian places. This allows us, in a third place, to define adequate notions of \emph{perfection} and \emph{eutaxy} for Hermite--Humbert forms and to demonstrate a theorem \`a la Voronoï. Eventually we bring forth some easy relations, upper bounds and computations relative to the Hermite constants.

\section{Representations and heights}
\subsection{Conventions}

In the sequel, an integer $n$ is fixed and the algebraic group we shall consider will always be the general linear group $G = \GL_n$.

\subsubsection{Global field} The letter $k$ refers to a \emph{number field}, that is an algebraic extension of $\Q$, of degree $d=r_1+2r_2$, where $r_1$ counts its real embeddings and $r_2$ counts its pairs of complex embeddings. Sometimes, $r$ may designate $r_1+r_2$. The embedings of $k$ into $\R$ or $\C$ are denoted by $(\si_j)_{1 \le j \le r}$, the $r_1$ first embeddings being real, the $r_2$ last embeddings being complex. The ring of integers of $k$ shall be written $\ido_k$ or simply $\ido$. The field $k$ encompasses $h$ ideal classes, the representative $\ida_1 = \ido$, $\ida_2$, \dots, $\ida_h$ of which we fix once for all. The norm of an ideal shall be denoted by $\norm(\ida)$.

\subsubsection{Local fields} The set of the places of $k$ is denoted by $\val$ and divides up into two parts, the set of archimedian or infinite places, denoted by $\val_{\infty}$ and the set of ultrametric or finite places, denoted by $\val_f$. The completion of $k$ (of $\ido$ respectivelly) at the place $v$ (where $v \in \val$) is denoted by $k_v$ ($\ido_v$ respectivelly). We shall call $d_v$ the local degree $[k_v : \Q_v]$. The completion $k_v$ is equiped with two absolute values : the absolute value  $\| \cdot \|_v$ which is the unique extension of either the absolute value of the real field $\Q_{\infty} $ when $v$ is an archimedean place or the one of the $p$-adic field $\Q_p$ when $v$ divides $p$ (that is $\|p \|_v = p^{-1} $), and the normalised absolute value $| \cdot |_v = \| \cdot \|^{d_v}$, which offers the benefit of satisfying the product formula, {\it i.e.} the equality $\prod_{v \in \val} |\al|_v = 1$ holds for any  $\al \in k^{\times}$.

\subsubsection{Partition and related items} The letter $\la$ shall always refer to a \emph{partition} of any integer $m$, which we shall note down by $\la \vdash m$. Within the borders of this article, we suppose additionnaly that a partition has always less then $n$ parts. Any partition can be depicted by a bar diagram (called Ferrer diagram) drawn in the first quadrant of the plane. The boxes which make up the diagram are indexed by their ``cartesian coordinates'', the most South--West box being the box $(1,1)$. The symbol $*$ pertains to the transpose partition $\la^*$, the diagramm  of which is by definition the symmetric with respect to the first bissector line of the diagramm of $\la$. The letters $s$ and $t$ refer to the width and the height of the Ferrer diagramm.
\begin{exem} Let $\la = (4,1)$ be the partition $5=4+1$, its diagramm is  
$$\la = \begin{young} 
  \cr
 & & & \cr \end{young}  $$
and the conjugate partition is $\la^* = (2,1,1,1)$. Here $s=4$ and $t=2$.
\end{exem}
 
To such a partition $\la$ is associated a character $\chi_{\la}$ defined on the torus $(k^{\times})^n$ by $\chi_{\la} : (x_1 , \dots, x_n) \in (k^{\times})^n \mapsto (x_1^{\la_1} x_2^{\la_2}  \dots  x_n^{\la_n}) \in k^{\times} $.

When $M$ is a real (respectivelly complex) vector or square matrix,  $\transp{M}$ is the transpose (respectivelly transconjugate) vector or matrix.

\subsubsection{Hermite--Humbert forms} We also recall the the space of Hermite--Humbert forms $\fhh$ is by definition the space 
$$\fhh = (\mathscr{S}_n^{>0})^{r_1} \times (\mathscr{H}_n^{>0})^{r_2}$$
where $\mathscr{S}_n^{>0}$ designates the set of determinant 1 symmetric positive definite matrices and $\mathscr{H}_n^{>0}$ the set of determinant 1 Hermitian positive definite matrices. (Depending on the authors, the condition concerning the determinant is not always retained but gives rise here to a convenient normalisation.) In the sequel, we may confuse without more precision a quadratic form on $k^n$ and the matrix which represents it in the canonical basis. Furthermore, the letter $\mathcal{I}$ means the $r$-tuple of identity matrices : $\mathcal{I}= (\I_n)_{1 \le j \le r} \in \fhh$.

\subsection{Irreducible representations of the general linear group}

To a partition $\la$ and a character $\chi_\la$ on the torus are classically associated a vector space  $\Sc_{\la}(k^n)$  and a representation $\pi_\la$ of the group $\GL_n(k)$, {\it i.e.} an action of the group $\GL_n(k)$ on the space  $\Sc_{\la}(k^n)$. The space $\Sc^{\la}(k)$ is sometimes called \emph{Weyl module} or \emph{Schur module}. We recall two equivalent constructions of it. For further details, one can consult \cite{fu} from which are excerpted  the following two very explicit contructions.

\subsubsection{Description of the Schur module by tableaux of vectors}

The cartesian product $E^{\times m}$ of a set $E$ is generally denoted in line by $E\times E \times \dots \times E$. In the sequel, we shall index each component of the product by one of the boxes of a partition $\la \vdash m$ and we shall write $E^{\times \la}$ to strengthen visually this convention. The first definition of  $\Sc^{\la}(k^n)$ leans upon the universal property described below. This definition will be held to represent elements of $\Sc^{\la}(k^n)$. 
  
\begin{defi}  Let $A$ be a commutative ring. For any $A$-module $E$, the Schur module $\Sc^{\la}(E)$ is the $A$-module equiped with a projection morphism $\rho_{\la} : E^{\times \la} \to \Sc^{\la}$ such that for any map $\va$ : $E^{\times \la} \to F$ from $E^{\times \la}$ to a $A$-module $F$, enjoying the following properties :
\begin{enumerate}
\item $\va$ is multilinear,
\item the restriction of $\va$ to any column of $\la$ is alternate,
\item for any pair of columns, for any choice of $p$ positions in the rightmost column, for any $v \in E^{\times \la}$, 
$$\va(v) = \sum_{w} \va(w)$$ 
where the sum concerns all the $w \in E^{\times \la}$ obtained from $v$ by flipping $p$ coefficients fixed in advance in the rightmost column with any $p$ coefficient in the leftmost column in an order preserving way within each column.
\end{enumerate}
there exists a unique homomorphism of $A$-modules $\tilde{\va}$ such that for any $m$-tuple of vectors  $v \in E^{\times \la}$, $\va(v) = \tilde{\va}(\rho_{\la}(v))$.\\
This construction can be summed up by the following commutative diagram


\centerline{\includegraphics{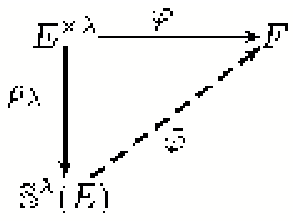}}
\end{defi}

The Schur module $\Sc^{\la}(k^n)$ can be realised as the quotient of $E^{\times \la}$ by some relations. We refer to \cite{fu} for more details and examples about this definition.

\begin{rema} When the partition $\la \vdash m$ admits only one part (horizontal diagram), the Schur module $\Sc^{\la}(E)$ is the space of the $m$-th symmetric powers $\text{Sym}^m(E)$ ; whereas when $\la \vdash m$ admits $m$ parts all equal to 1 (vertical diagram), the Schur module $\Sc^{\la}(E)$ is the space of the $m$-th exterior products $\bigwedge^m(E)$.
\end{rema}

One can see elements of the Schur module $\Sc^{\la}(E)$ as linear combination of diagrams of shape $\la$ inscribed with vectors from $E$. Yet, for a given element of $\Sc^{\la}(E)$, this script is not necessarily unique. For example, as well as $x \land y$ and $x \land (x+y)$ represent the same projection of the pair $(x,y) \in E^2$ in $\bigwedge^2(E)$, the notation in the shape of an inscribed diagram 
$$\begin{young}
t \cr 
x & y & z \cr
\end{young}  $$
stands for the projection of the quadruplet $(t,x,y,z) \in E^4$ in the module $\Sc^{\hoogte=2pt 
 \breedte=3pt
 \begin{young}
 \cr
 & &  \cr
\end{young} 
 \hoogte=15pt  
 \breedte=18pt}(E)$.

\begin{prop} When $E$ is a \emph{free} $A$-module, the Schur module $\Sc^{\la}(E)$  is also a\emph{free} $A$-module and the theory even affords us with a basis. Namely, this basis is made out the vectors 
\newdimen\breedte   \breedte=32pt 
$$e_T = \begin{young}
e_{T_{1,t}} \cr
e_{T_{1,t-1}} & \cr
& & \cr
e_{T_{1,1}} & e_{T_{2,1}} &  & e_{T_{s,1}} \cr 
\end{young} $$
\newdimen\breedte   \breedte=18pt 
where $(e_i)_{1 \le i \le n}$ is a basis of $E$ and $T$ a \emph{Young tableau}, {\it i.e.} a diagram of shape $\la$, inscribed with numbers chosen in $\Np{1}{n}$, the inscription of which are strictly increasing along the column and non-decreasing along the rows.
\end{prop}
\begin{exem} Let us set $E = \R^3$ and take $B = (\vec i , \vec j, \vec k)$ a basis of $\R^3$, then a basis of the Schur module $\Sc^{\hoogte=2pt 
 \breedte=3pt
 \begin{young}
 \cr
 &   \cr
\end{young} 
 \hoogte=15pt  
 \breedte=18pt}(\R^3)$ is given by
$$ \hoogte=15pt  
 \breedte=18pt \left\{ \, \begin{young} \vec j \cr \vec i & \vec i \cr \end{young}\, , 
\begin{young} \vec k \cr \vec i & \vec i \cr \end{young}\, ,
\begin{young} \vec j \cr \vec i & \vec j \cr \end{young}\, ,
\begin{young} \vec k \cr \vec i & \vec j \cr \end{young}\, ,
\begin{young} \vec j \cr \vec i & \vec k \cr \end{young}\, ,
\begin{young} \vec k \cr \vec i & \vec k \cr \end{young}\, ,
\begin{young} \vec k \cr \vec j & \vec j \cr \end{young}\, ,
\begin{young} \vec k \cr \vec j & \vec k \cr \end{young} \right\}$$
\end{exem}
\begin{rema} In the case of exterior products or symmetric powers, this results is nothing more then the well--known following fact : the vectors $e_{i_1} \land e_{i_2} \land \dots \land e_{i_m}$ constitute a basis of $\Sc^{\hoogte=2pt 
 \breedte=3pt
 \begin{young}
 \cr
 \cr 
 \cr
\end{young} 
 \hoogte=15pt  
 \breedte=18pt}(E) = \bigwedge^m(E)$ for $i_1 < \dots < i_m$ and the vectors $e_{i_1} \odot e_{i_2} \odot \dots \odot e_{i_m}$ constitute a basis of $\Sc^{\hoogte=2pt 
 \breedte=3pt
 \begin{young}
 &  &   \cr
\end{young} 
 \hoogte=15pt  
 \breedte=18pt}(E) = \text{Sym}^m(E)$ for $i_1 \le \dots \le i_m$. 
\end{rema}

\subsubsection{Description of the Schur module by spaces of polynomials}

In the sequel, we shall also need a second description of the Schur module, adapted to define a scalar product. Indeed, the Schur module $\Sc^{\la}(E)$ can be realised as a certain subspace of polynomials in many variables. Let $Z$ be the matrix of indeterminates
$$Z = \begin{pmatrix} z_{1,1}  & \dots & z_{1,n} \\
\vdots & & \vdots \\
z_{t, 1} & \dots & z_{t, n} 
\end{pmatrix} $$
and $k[Z]$ the space of polynomials in the $nt$ variables of $Z$. Beware that elements of $k[Z]$ are not polynomials of the matrix $Z$.

For integers $i_1$, \dots, $i_r$, let $D_{(i_1, \dots, i_r  )}$ be the following determinant
$$ D_{(i_1, \dots, i_r  )} =  \begin{vmatrix} z_{1,i_1} & \dots & z_{1,i_r} \\
z_{2,i_1} & \dots & z_{2,i_r} \\ 
\vdots & & \vdots \\
z_{r,i_1} & \dots & z_{r,i_r} 
\end{vmatrix} $$
When $T$ is a tableau inscribed with integers taken between $1$ and $n$, we call $\phi_T$ the polynomial 
$$\phi_T = \underbrace{ D_{(T(1,1) , \dots , T(1,\la_1^*)} }_{\text{relative to the 1st column}} 
\underbrace{ D_{(T(2,1) , \dots , T(2,\la_2^*)} }_{\text{relative to the 2nd column}} 
\dots
\underbrace{ D_{(T(s,1) , \dots , T(s,\la_s^*)} }_{\text{relative to the last column}} .
$$

The linear map $E^{\times \la} \to k[Z]$ defined by $e_T \mapsto \phi_T$ factors out through the universal property of the Schur module into an injective map $\Sc_{\la}(E) \to k[Z]$. This shows that $\Sc_{\la}(E)$ is isomorphic to the subspace $\D^{\la}(E)$ of $k[Z]$ spanned by the set of polynomials $\phi_T$ where $T$ runs through the set of all Young tableaux.

\begin{prop} The image $\D^{\la}(E)$ in $k[Z]$ of $\phi$ is isomorphic to the Schur module $\Sc_{\la}(E)$.
\end{prop}

We identify these two spaces by an isomorphism that we call again $\phi$.

\subsection{The representations of the general linear group} 

By functoriality of the construction of the Schur module, the standard action of 
$\GL_n(k)$ on $k^n$ determines an action $\pi_{\la}$ of the general linear group $\GL_n(k)$ on the Schur module $\Sc_{\la}(k^n)$. 

\begin{defi} We call $\pi_{\la}$ the action of the general linear group on the Schur module described on decomposed vectors by 
\be  \hoogte=15pt  
 \breedte=18pt \forall g \in \GL_n(k), \; \forall X = \begin{young}
x_2 \cr 
x_1 & y_1 & \dots \cr 
\end{young} \, , \qquad \pi_{\la}(g).X =
\newdimen\breedte   \breedte=30pt  
 \begin{young} 
g(x_2) \cr 
g(x_1) & g(y_1) & \dots \cr \end{young}  .\ee
\newdimen\breedte   \breedte=18pt 
\end{defi}

The vector $e_{U(\la)}$ of the Schur module $ \Sc^{\la}(k^n)$ where $U(\la)$ stands for the Young tableau
\be  \hoogte=20pt  
 \breedte=18pt U(\la) = \begin{young}
\la_1^* \cr
\, \cdots & \la_2^* \cr
2 & \, \cdots & \, \cdots & 2 \cr
1 & 1 & 1 & 1 & 1 \cr
\end{young} \label{e:12} \ee
is invariant under the action of the subgroup of unipotent upper triangular matrices (which comprise only ones on the diagonal). This vector is said to direct the line of the \emph{highest weight vectors} of $\pi_{\la}$. It is invariant under the action of the parabolic subgroup $P_{\la}$ of $G$ which stabilises the flag $\mathscr{D}_0 = \{ \text{span} (e_1, \dots, e_{\la_\ell^*} ) \}_{1 \le \ell \le s}$.

Under the action of a diagonal matrix $h$, the vector $e_{U(\la)}$ is simply multiplied by $\chi_{\la}(h)$. Thus, the character of the torus $\chi_{\la}$ is called the \emph{weight} of $\pi_{\la}$.

Let us notice eventually that the isomorphism $\phi$ turns the action of the general linear group $\GL_n(k)$ on the Schur module $\Sc_{\la}(k^n)$ into an action of right multiplication on the subspace $\D^{\la}(k^n)$  of the polynomials $ k[Z]$. In other words, if $X $ is a vector of $\Sc^{\la}(k^n)$, if the polynomial $P(Z)$ is its image by $\phi$ in $\D^{\la}$, and if $g$ belongs to the group $ \GL_n(k)$, then
$$ \phi ( \pi_{\la}(g).X ) = P(Z.g)  $$

\begin{prop} Up to isomorphisms, there are no other irreducible polynomial representations of $\GL_n(k)$ than the representations $(\pi_{\la}, \Sc^{\la}(k^n))$ described above.
\end{prop}

\subsection{The height}

We define now a specific multiplicative height $H$ on   $\Sc^{\la}(k^n)$ by its local factors :
\be H(x) = h_0 \prod_{v \in \val} H_v(x) \ee
where $h_0$ is a normalisation constant and $H_v$ is a norm on the $k_v$-vector space $\Sc^{\la}(k_v^n)$ compatible with the absolute value $|\cdot|_v$.

Contrary to the generally used heights, we do not define the local norm with respect to a unique foregone basis, because the heights obtained in this way are not invariant under the action of the compact group 
$K_n (\A_k) = \displaystyle \prod_{v \in \val_{\infty}} O_n(k_n) \, \times \, \prod_{v \in \val_f} \GL_n(\ido_v)$. The rest of this paragraph is devoted to the description of the height that the action of this group leaves invariant.

\subsubsection{Infinite places :}

Let $v \in \val_{\infty} $ be an infinite place. There exists on the space of polynomials $\D^{\la}(k_v^n)$ a Euclidean (or Hermitian) scalar product for which two monomials are orthogonal when they are distinct and for which the scalar square of the monomial $z_{k_{1,1}}^{\al_{1,1}} z_{k_{1,2}}^{\al_{1,2}} \cdots z_{k_{t,n}}^{\al_{t,n}}$ is $\al ! = \al_{1,1} ! \al_{1,2} ! \cdots \al_{t,n} !$. It can be checked that the scalar product can also be described by
$$\forall P,Q \in k_v[Z], \qquad \langle P,Q \rangle = L_0 \left(  P(\underline{\pa}) \, \overline{Q}(Z ) \right) .$$ 
where the matrix $\underline{\pa}$ is the matrix of the derivation operators 
$$\underline{\pa} = \begin{pmatrix} \pa_{z_{1,1}}  & \dots & \pa_{z_{1,n}} \\
\vdots & & \vdots \\
\pa_{z_{\la_1^*, 1}} & \dots & \pa_{z_{t, n}} 
\end{pmatrix} $$
and $L_0$ is the operator that evaluates the polynomials in zero.

\begin{exem} Consider the two polynomials $P(Z) = z_{1,1} + z_{2,2}$ and $Q(Z) = z_{1,1}^2 + z_{2,2}$, then we have $P(\underline{\pa}) \, \overline{Q}(Z ) = (\pa_{1,1} + \pa_{2,2})( z_{1,1}^2 + z_{2,2}) = 2 z_{1,1} + 1$, and their scalar product is just $\langle P,Q \rangle = 1$.
\end{exem}

\begin{defi} We set in that case 
$$H_v(X) = \ds \frac{\dim \left( \Sc^\la(k ^n) \right) }{n!} \langle X,X\rangle ^{d_v/2},$$
for any vector $X $ of the Schur module $\Sc^{\la}(k_v^n) \overset{\phi}{\simeq} \D^{\la}(k_v^n)$.
\end{defi}

\begin{lemm} \label{l:12} When $v$ is a real place (respectively when $v$ is a complex place), the action of the orthogonal group $O_n(\R)$ (respectively the unitary group $U_n(\C)$) on $\Sc^{\la}(k_v^n)$ is orthogonal (respectively unitary).
\end{lemm}

\begin{proof} We handle the real case ; the complex case can be demonstrated similarly.
Let us take $\om \in O_n(\R)$, and $P,Q \in \D^{\la}(\R^n)$. Since the action of an isometry commutes with the derivation with respect to a polynomial of differential operators, we get the equality
\be [(\om.P) (\underline{\pa})][(\om.Q) (Z  )] =  \om [P(\underline{\pa})][Q ( Z)] \label{e:14} \ee

Indeed, by linearity, it suffices to show this result on monomials. We proceed by induction. Let us denote by $\om = (( \om_{i,j}))_{1 \le i , j \le n}$ the coefficients of the orthogonal element $\om$.
\begin{itemize}
\item If $P$ and $Q$ are reduced to one variable, let us say $P=z_{i,j}$ and $Q=z_ {i',j'}$, then $ \om. P(\underline{\pa}) = \sum_{k=1}^n \om_{k,j} \pa_{z_{i,k}} $ and $\om . Q = \sum_{k=1}^n z_{i',k} \om_{k,j'} $. Thereby
$$(\om. P(\underline{\pa})) (\om . Q) = \left\{ \begin{matrix} 0& \text{ if } i \neq i' \text{ since the derivatives are all zero}\\ \ds \sum_{k=0}^n \om_{k,j} \om_{k,j'} = 0 & \text{ if } i=i' \text{ and } j \neq j' \text{ by orthogonality}
\\ \ds \sum_{k=0}^n \om_{k,j} \om_{k,j'} = 1  & \text{ if } i=i' \text{ and } j = j' \text{ by orthogonality}\end{matrix} \right. .$$ 
\item If $P = z_{i,j}$ is of degree one, and if we suppose having demonstrated the relation for any monomial $Q$ of degree less than or equal to $n$, let $Q$ be a monomial of degree $n+1$ that we write $Q=z_{i',j'} Q_1$, 
\Bea (\om. P(\underline{\pa})) (\om . Q) &=& ((\om. P(\underline{\pa}))  (\om. z_{i',j'})) (\om.Q_1) + z_{i',j'} (\om. P(\underline{\pa})).Q_1 \\
&=& \om .( P(\underline{\pa}) z_{i,j}) + (\om.z_{i',j'}) \, \om P(\underline{\pa}) . (\om.Q_1) \Eea
Using the induction hypothesis, $(\om. P(\underline{\pa})) (\om . Q) = \om.(P(\underline{\pa})  Q)$. 
\item Eventually, let us suppose the relation demonstrated for any monomial $Q$ and any monomial $P$ of degree less or equal to $n$. Let $P$ be a monomial of degree $n+1$, let us say $P = z_{i,j} P_1$, then, 
\Bea (\om. P(\underline{\pa})) (\om . Q) &=& (\om. P_1(\underline{\pa})) [ (\om. \pa_{z_{i,j}}) (\om . Q) ] \\
&=&  (\om. P_1(\underline{\pa})) [ \om . (\pa_{z_{i,j}}  Q)] \\
&=& \om. (P_1(\pa) \pa_{z_{i,j}}  Q) \\
&=& \om. (P(\pa)   Q) . \Eea
which completes the proof of the relation (\ref{e:14}).
\end{itemize}
When this relation is evaluated in zero, we get as expected
$$ \langle \om.P, \om.Q \rangle =  \langle P, Q \rangle .$$
\end{proof}

\begin{defi}The notation $\De_{\ell}$ symbolises henceforth the principal minor of order $\ell$ of any matrix (of size greater than or equal to $\ell \times \ell$).
\end{defi}
We express by its medium the local height $H_v(\pi_{\la}(g) . e_{U(\la)})$ in a pleasant form.

\begin{prop} Let $\pi_{\la}$ be the representation defined above, $e_{U(\la)}$ the highest weight vector defined by (\ref{e:12}) and $g \in \GL_n(k)$, the local height of $\pi_{\la}(g) . e_{U(\la)}$ is equal to the product of the following minors :
\be H_v ( \pi_{\la}(g) . e_{U(\la)} ) = \left( \De_{\la_1^*} (\transp{g} g) \De_{\la_2^*} (\transp{g} g) \dots \De_{\la_s^*} (\transp{g} g) \right)^{d_v / 2} \, H_v (e_{U(\la)} )\ee
where $v$ is an archimedean place.
\label{p:13}
\end{prop}

\begin{proof} The place $v$ being fixed, the group $\GL_n(k)$ can by embedded into $\GL_n(\C)$. Take $g \in \GL_n(k)$, $g$ can be written according to Iwasawa decomposition into the product $g = u d b$ where $u$ is a unitary matrix, $d = \mathrm{Diag}(\al_1, \dots, \al_n)$ is a diagonal matrix and $b$ a unipotent upper triangular matrix.

\begin{itemize}
\item We know already that the action of $U_n(\C)$ does not alter the norm (lemma \ref{l:12}). Thus  $H_v ( \pi_{\la}(g) . e_{U(\la)} ) =  H_v ( \pi_{\la}( db) . e_{U(\la)} )$. On the other hand, for any $l$, $\De_{l}(\transp{g}g) = \De_{l}(\transp{b} \transp{d} d b)$. Thus we can assume that $u= \Id$. \\

\item Besides, as $e_{U(\la)}$ is the highest weight vector, $\pi_{\la}(d b ) . e_{U(\la)} = \chi_{\la} (d) e_{U(\la)} = \left( \prod_{i=1}^n \al_i^{\la_i} \right) e_{U(\la)} $, hence
$$H_v\left( \pi_{\la}(d b ) . e_{U(\la)}\right) = \left| \prod_{i=1}^n \al_i^{\la_i} \right| H_v ( e_{U(\la)})  $$
But $db$ is an upper triangular matrix. In particular, $db$ can be written block by block in the shape $\begin{pmatrix} m_1 & * \\ 0 & m_2 \end{pmatrix}$ with sizes $\ell$ and $n-\ell$ and thus $\transp{(db)} d b = \begin{pmatrix} \transp{m_1} m_1 & * \\ * & * \end{pmatrix}$, which justifies the equality $\De_{\ell} ( \transp{(db)} d b)  = \|\De_{\ell} ( d b)\|^2$. Accordingly, $\De_{\ell} ( \transp{b} \transp{d} d b) = |\De_{\ell} (d)|^{2/d_v} = \prod_{i=1}^{\ell} |\al_i|^{2/d_v}$.
But, 
$\prod_{\ell=1}^s \De_{\la_{\ell}^*} (\transp{g} g)  = \prod_{i=1}^{n} |\al_i^{\la_i}|^{2/d_v} $ since the term $\al_i$ appears exactly $\la_i$ times among the determinants $(\De_{\la_{\ell}^*} (\transp{g} g))_{1 \le \ell \le s}$. This last equality ends up the proof the proposition. \end{itemize} \end{proof}

\subsubsection{Finite places :} We know already a basis of $\Sc^{\la}(k_v^n)$, given by the vectors $e_T$ where the $T$ are \emph{Young tableaux}, that is a diagram inscribed with increasing integers along columns and non-decreasing integers along lines, all coefficients being in $\Np{1}{n}$. 
\begin{defi}For a finite place $v$, we set the local norm  $H_v(x)$  as the maximum of the absolute value of the coordinates of $x$ in this basis.
\end{defi}
 In this way, we have in addition that the Schur module of the cartesian product of the ring of integers is given by the following equality
$$\Sc^{\la}(\ido_v^n)=\{ X \in \Sc^{\la}(k_v^n) ; H_v(X) \le 1 \}$$
which motivates the choice of our local norm.

\subsubsection{Normalisation :} The constant $h_0$ is chosen such that $H(e_{U(\la)})=1$. Examples suggest that $h_0$ should be equal to $1$.

\begin{rema} In the case when $\pi_{\la}$ is the representation on the space of exterior powers of $k^n$, the height we get is the one studied as example 1 in \cite{wa1} page 43, since the basis used for finite places is also orthonormal for the scalar products of the infinite places. As a result, the constant studied in this article matches with the constant defined by J. Thunder in \cite{thu2}.
\end{rema}

\section{A second formulation with Hermite--Humbert forms}
\subsection{Adelic definition of the constant}

Let us recall before we start the definition of the constant we intend to study. From the adelic view point, this constant is given by
\be \ga_{n, \la} = \max_{g \in \GL_n(\A_k)^1} \min_{\ga \in GL_n(k)} H( \, \pi_{\la}(g\ga) \,\, e_{U(\la)} \, )^{2/d} \ee
which is the constant defined by the equality  (\ref{eq:1}) associated to the general linear group $\GL_n$, to the height suggested in the previous section and to the representation $\pi_{\la}$ also previously described.

\begin{defi} Let us denote by $\Sc_{\sharp}^{\la}(k^n)$ the set of non zero vectors $X \in \Sc^{\la}(k^n)$ that have the shape

$$  \hoogte=15pt  
 \breedte=18pt X = \begin{young} x_t \cr
 : \cr
 x_2 & x_2  \cr x_1  & x_1  & . \, .  & x_1   \cr \end{young},$$
\newdimen\hoogte    \hoogte=15pt
with $(x_{\ell})_{1 \le \ell \le t} \subset k^n$. We shall call these particular vectors \emph{flag vectors}, in reference to their characterising, up to a multiplicative constant, the flag of nested subspaces spanned by  $(x_i)_{1 \le i \le \la_k^*}$ when $k$ varies between $1$ and $s$.
\end{defi}

\begin{defi} When an automorphism  $A \in \GL_n(\A_k)$ is fixed, it is possible to define a  \emph{twisted height} $H_A$ by the formula $H_A(X) = H(\pi_{\la}(A)(X))$ which applies to any vector  $X \in \Sc^{\la}(k^n)$ as well as any flag $\mathscr{D}$ represented by a flag vector $X$. The \emph{Hermite constant} $\ga_{n, \la}$ is then the smallest constant $c$ such that there exists a flag  $\mathscr{D}$ fulfilling the inequality
$$ H_A(\mathscr{D}) \le c^{1/2} |\det ( A) |_{\A_k}^{|\la|/n} $$
for any automorphism $A \in \GL_n(\A_k)$.
By homogeneity of the inequality, the constant $\ga_{n, \la}$ is still the smallest constant $c$ such that there exists a flag $\mathscr{D}$ satisfying
$$ H_A(\mathscr{D}) \le c^{1/2} $$
for any automorphism $A \in \GL_n(\A_k)$ of determinant of norm $|\det(A)|_{\A_k} = 1$. \end{defi}

\begin{rema}
As it can be acknowledged, the constant $\ga_{n, \la}$ depends only the heights of flag vectors. A more sophisticated way to describe the heights could have been to consider the one obtained  through the embedding (\ref{e:pldrap}) with a section and a metrised line bundle on the flag variety $P_\la \backslash G$ itself. This latter construction turns out to define the same heights as we did.

To state a precise result, we shall use the following notation. The flag $P_\la \backslash G$ is actually the collection of nested subspaces the dimension of which is equal to one of the $\la_i^*$ for some integer $i$ between $1$ and $s$. The letter $a_d$ will refer to number of occurence of the integer $d$ among the parts of the partition $\la^*$, or in other words the number of times the subspace of dimension $d$ is repeated.

Using the maps that send one of the nested subspace of dimension $d$ of the flag into the grassmanian space $\mathscr{G}_d (k^n) $ of the corresponding dimension $d$, the maps that send a subspace $V$ into $\bigwedge^{\dim V} V$, and the Veronese embeddings that send $V$ into $\text{Sym}^a V$, we dispose of the following chain of embeddings :
$$ P_\la \backslash G \subset \prod_{d \in  \{\la_i^*, \, 1 \le i \le s \}}  \mathscr{G}_{d} (k^n)  \subset \prod_{d \in  \{\la_i^*, \, 1 \le i \le s \}} \mathbb{P} \left( \bigwedge^{d} k^n\right) \subset \cdots$$
$$\cdots \subset \prod_{d \in  \{\la_i^*, \, 1 \le i \le s \}} \mathbb{P} \left( \text{Sym}^{a_d}  (\bigwedge^d k^n) \right) \subset \mathbb{P} \left( \bigotimes_{d \in  \{\la_i^*, \, 1 \le i \le s \}}  \text{Sym}^{a_d}  ( \bigwedge^d k^n) \right) $$

It is classical to attach to any grassmanian space $\mathscr{G}_d $  the line bundle $V \mapsto (V, \bigwedge^{\dim V} V)$ whereupon the metric is give using the standard metric for the ultrametric places and the Fubini--Study metric for the archimedean places (see section 2.9 in \cite{bg}). Taking the $a$-th tensor product gives a metrised line bundle on $\text{Sym}^a$ and by tensoring again, we get a metrised line bundle of $\mathbb{P} \left( \bigotimes_{d \in  \{\la_i^*, \, 1 \le i \le s \}}  \text{Sym}^{a_d}  ( \bigwedge^d k^n) \right)  $ which we denote $\mathcal{O}(a_1, \dots, a_t)$. 

The metrised line bundle $\mathcal{O}(a_1, \dots, a_t)$ on $P_\la \backslash G$  is equal to the line bundle $\mathcal{O}(1)$ obtained through the embedding of the flag variety into $\mathbb{P}(\Sc^\la (k^n))$ (see for more detail section 9.3 of \cite{fu}) and the metric we defined previously. See also description in \cite{thu1}.
\end{rema}

\subsection{Definition based on Hermite--Humbert forms} 

We first define the evaluation of a Humbert form at a flag vector.

\begin{defi} Let $\mathcal{A} = (A_j)_{1 \le j \le r_1+r_2} \in \fhh$ be a Hermite-Humbert form and take a flag vector $X \in \Sc_{\sharp}^{\la}(k^n)$, then the notation $\mathcal{A}[X]$ symbolises the evaluation of $\mathcal{A}$ in $X$, which means :
\Bea \mathcal{A}[X] &=& \prod_{j=1}^{r_1+ r_2} \prod_{\ell=1}^t   \Big(\De_{\la_{\ell}^*} \big(\transp{[x_1, \dots, x_s]} A_j [x_1, \dots, x_s] \big )\Big)^{d_j} \\
&=& \prod_{\ell=1}^t \left( \prod_{j=1}^{r_1} \De_{\la_{\ell}^*} \big(\transp{[x_1, \dots, x_s]} A_j [x_1, \dots, x_s] \big) \; \prod_{j=r_1+1}^{r_1+r_2} \Big(\De_{\la_{\ell}^*} \big(\transp{[x_1, \dots, x_s]} A_j [x_1, \dots, x_s] \big)\Big)^2 \right) \Eea
\end{defi}

\subsubsection{Minimum and determinant of a form} When $L$ is a lattice of $k^n$, we know that $L$ admits a pseudo basis and can be decomposed into
\be \label{eq:21} L = \idc_1 u_1 \oplus \dots \oplus \idc_n u_n \ee
where the $(\idc_i)_{1 \le i \le n}$ are fractional ideals and the system $(u_i)_{1\le i \le n}$ is a vectorial basis of $k^n$. Following \cite{cou2}, we define the determinant of a Hermite--Humbert quadratic form  $\mathcal{A} = (A_j)_{1 \le j \le r_1+r_2}$ with respect to the lattice $L$ by the following product of the determinants computed in the system of vectors $(u_i)_{1 \le i \le n}$ :
\be {\det}_L (\mathcal{A}) = \norm( \idc_1 \idc_2 \dots \idc_n ) \prod_{j=1}^{r_1+r_2}  \big( {\det}_{(u_1, \dots, u_n)} (A_j) \big)^{d_{v_j}}   \ee

We define also, again with respect to a lattice $L$ of $k^n$ and a vector $X$ of the Schur module  $\Sc^{\la}(k^n)$, the fractionnal ideal $\idA_X^L$ according to the formula :
\be (\idA_X^{L})^{-1} = \left\{ \al \in k, \; \al X \in \Sc^{\la}(L) \right\} \ee
In the case when $\la$ is the partition of 1, the ideal $\idA_X^{L}$ is nothing else than the greatest common divisor of the coordinates of $X$ (in that case $X$ is simply a vector of $k^n$).

We define the minimum of a Hermite--Humbert form $\mathcal{A} \in \fhh$ relatively to the lattice $L$ as the quantity 
\be m_{L} (\mathcal{A}) = \min_{X \in \Sc_{\sharp}^{\la}(L)} \frac{\mathcal{A}[X] }{\norm(\idA_X^L)} .\ee

\subsubsection{A new Hermite constant} We call generalised Hermite invariant of the Hermite--Humbert form $\mathcal{A}$ relative to the lattice $L$ the number
\be \ga_L (\mathcal{A}) = \frac{m_L(\mathcal{A}) }{ \left({\det}_L \mathcal{A} \right)^{\frac{m}{n}} }.\ee
We eventually define the constant
\be \ga_L = \sup_{\mathcal{A} \in \fhh} \ga_L(\mathcal{A}) \ee

Remember that the Steinitz class of lattice means the ideal class $\idc = \idc_1 \idc_2 \cdots \idc_n$ (with the notation of equality \ref{eq:21}). If $L$ and $L'$ are two lattices with the same Steinitz class, then they are isomorphic and as a result, the constants $\ga_L$ and $\ga_{L'}$ coincide. The canonical basis of $k^n$ being $(e_i)_{1 \le i \le n}$, let us fix the following representatives for each class of lattice
\be \forall \, 1 \le \iota \le h, \quad L_{\iota} = \ido e_1 \oplus \dots \oplus \ido_{n-1} e_{n-1} \oplus \ida_{\iota} e_n \ee
Let us denote by $\hat{\ga}_{n,\la}$ a new Hermite constant
\be \hat{\ga}_{n,\la} = \max_{L \text{ lattice of } k^n} \ga_L = \max_{1 \le \iota \le h} \ga_{L_{\iota}} \ee

\subsection{Equivalence of the definitions}

\begin{prop}
The constant $\hat{\ga}_{n, \la}$ is equal to the constant $\ga_{n, \la}$ introduced by T. Watanabe.
\end{prop}

\begin{proof}
The group $\GL_n(\A_k)$ admits a double cossets decomposition (see \cite{pla}),
\be \GL_n(\A_k) = \bigsqcup_{\iota = 1}^{h} \GL_n(\A_{k,\infty}) \, \la_{\iota} \, \GL_n(k) \ee
 obtained by the action of $\GL_n(k)$ on the right and the action of $\GL_n(\A_{k, \infty})$ on the left, where
\be \GL_n(\A_{k, \infty}) = \prod_{v | \infty} \GL_n(k_v) \times \prod_{v\in \nmid \infty} \GL_n(\ido_v) .\ee 
Here come the details :\\
For $1 \le \iota \le h$, the ideal $\ida_{\iota} \otimes \ido_v$ localised in $k_v$ ($v$ finite place) of the representative $\ida_{\iota}$ of an ideal class becomes a principal ideal and takes the shape of $a_{\iota, v} \ido_v$ where $a_{\iota, v} \in k_v$. Let us denote $a_{\iota} \in \A_k$ the ad\`ele  we get by filling out with $a_{\iota,v} = 1$ for archimedian places $v$. Let $\la_{\iota}$ be the diagonal matrix  $\text{Diag}(1, \dots, 1, a_{\iota}^{-1})$ of $\GL_n(\A_k)$. Let us notice that 
\be \la_{\iota} L_{\iota} = L_1 \qquad \text{and} \qquad |\det \la_{\iota}|_{\A_k} = \norm(\ida_{\iota}) \ee
Then the decomposition of  $\GL_n(\A_k)$ proves to be
$$ \GL_n(\A_k) = \bigsqcup_{\iota = 1}^{h} \GL_n(\A_{k,\infty}) \, \la_{\iota} \, \GL_n(k) $$

The value of $\displaystyle \min_{\ga \in \GL_n(k)}  H \left( \pi_{\la}(g\ga) e_{U(\la)} \right)^2 $ does not depend actually of the class of $g \in \GL_n(\A_k)$ modulo $\GL_n(k)$. Moreover,when $\ga$ runs through $\GL_n(k)$, $\pi_{\la}(\ga) e_{U(\la)}$ describes all the decomposed vectors of $\Sc_{\sharp}^{\la}(k^n)$. Thus, the constant $\ga_{n, \la}$ can be rewritten
\Bea \ga_{\la} &=& \max_{g \in \GL_n(\A_k)/\GL_n(k)} \min_{X \in \Sc_{\sharp}^{\la}(k^n)} \frac{H(\pi_{\la}(g)X)}{ |\det g|_{\A_k}^{\frac{2m}{n}} } \\
&=& \max_{g \in \GL_n(\A_{k, \infty})} \max_{1 \le \iota \le h} \min_{X \in \Sc_{\sharp}^{\la}(k^n)} \frac{H(g\la_{\iota}.X)}{ |\det g\la_{\iota}|_{\A_k}^{\frac{2m}{n}} }
\Eea

Because of the product formula, it can be assumed 	that $X$ runs  only through the set $\Sc_{\sharp}^{\la}(L_{\iota})$. Because of the invariance of the local heights under the action of, $\GL_n(\ido_v)$,  $g$ can be restricted to $\GL_n(k_{\infty})=  \prod_{v | \infty} \GL_n(k_v)$.
To summarize :
$$\ga_{n, \la} = \max_{1 \le  \iota \le h} \max_{g \in \GL_n(k_{\infty})} \min_{X \in \Sc_{\sharp}^{\la}(L_{\iota)}} \frac{H(g \la_i. X)}{ |\det g|_{\A_k}^{\frac{2m}{n}} \left(\norm(\ida_{\iota}) \right)^{\frac{2k}{n}} }  $$
Let us associate to any $g \in \GL_n(k_{\infty})$ the Hermite--Humbert form $\mathcal{A} = \transp{g}g$. The product of the archimedian local heights $\prod_{v | \infty} H_v( g_v \la_{\iota,v} X )$ is equal to $\mathcal{A}[X]$ according to proposition (\ref{p:13}) and the fact that $\la_{\iota,v} = \Id$ when $v$ is an archimedian place. The product of ultrametric local heigths $\prod_{v \nmid \infty} H_v( \la_{\iota,v} X ) $ is exactly $\frac{1}{\norm(\idA_X^{\iota} )}$.\\
Indeed, let us decompose $X = \sum_T X_T e_T$ (where $T$ runs among Young tableaux) in the canonical basis of $\Sc^{\la}$. The action of $\la_{\iota}$ in that basis boils down to multiplying $e_T$ by $a_{\iota}^{- \# \{ (i,j) ; \, T_{i,j} = n \}}$, whence the formula
$$H_v(\la_{\iota} X) = \max_T |a_{\iota}^{- \# \{ (i,j) ; \, T_{i,j} = n \}}|_v |X_T|_v .$$
But $\al X$ belongs to $\ds \Sc^{\la}(L_{\iota}) = \bigoplus_T \ida^{\# \{ (i,j) ; \, T_{i,j} = n \}} \, e_T$ if and only if $|\al X_T|_v \ge |\ida_{\iota}^{\# \{ (i,j) ; \, T_{i,j} = n \}}|_v$ for any finite place $v$. Therefore $|(\idA_X^{\iota})^{-1}|_v = \max_T \frac{|\ida_{\iota}^{\# \{ (i,j) ; \, T_{i,j} = n \}}|_v}{|X_T|_v} $. By bringing all the equalities together, we get the expected formula. Thereby,
\be \ga_{n, \la} = \max_{1 \le  \iota \le h} \max_{\mathcal{A} \in \fhh} \min_{X \in \Sc_{\sharp}^{\la}(L_{\iota)}} \frac{ \mathcal{A}[X]}{ \norm(\idA_X^{\iota}) ({\det}_{L_{\iota}} \mathcal{A})^{\frac{m}{n}}  }  = \hat{\ga}_{n, \la} \ee
as expected\end{proof}

\section{Eutaxy and perfection}

The classical Voronoi theory \cite{vor} is based on two properties of quadratic forms, perfection and eutaxy, which characterise extreme forms, {\it i.e.} those that are a local maximum of the Hermite invariant $\ga$.

In his article \cite{bav1}, Christophe Bavard has set forth a general framework we make ours here to propose appropriate definitions of eutaxy and perfection. The framework is as follows : the Hermite invariant of an object $p$ is defined as the minimum of the evaluation in $p$ of a familly of \og length functions \fg  $(f_c)_{c \in C}$, which are positive real-valued functions, defined on a space $V$ which parametrises the set of objects $p$ in mind. To redound on the classical case of Voronoi theory, $V$ has to be chosen as the space of determinant 1 symmetric definite positive matrices, the length functions are the functions $A \mapsto \transp{u} A u \in \R_+$  and are indexed by the set of non-zero vectors $u \in \Z^n$.  

In this geometrical framework, eutaxy and perfection of an object $p$ can be formulated in terms of properties on gradients of the length functions that attain the minimum ({\it cf. infra.}). Moreover, when the following condition $(\mathscr{C})$ is met, the Voronoi theorem holds, {\it i.e.} the Hermite invariant of $p$ is a local maximum if and only if $p$ is eutactic and perfect.

Let $S(p)$ be the set of indexes $\varsigma \in C$ of the length functions $f_\varsigma$ that are minimal in $p$. The condition $(\mathscr{C})$ can be stated as follows : 
\begin{quote} \og For any point $p$ of $V$, for any subset $T \subset S$, for any non zero vector $x$ orthogonal to the gradients $(\na f_\vartheta)_{\vartheta \in T} $, there exists a  $\mathcal{C}^1$  stalk of curve $c : [0, \eps[ \to V$ such that $c(0)=p$, $c'(0) = x$ and for any $\vartheta \in T$, $f_\vartheta(c(t)) > f_\vartheta (p)$ when $t \in ]0,\eps[$. \fg
\end{quote}

\subsection{Terms reformulation}

In our case, the set of Hermite--Humbert forms are parametrised by the space $\fhh$ embedded with a \emph{Riemannian structure}, and in particular the scalar product defined on the tangent space by :
\be \forall \mathcal{X}, \, \mathcal{Y} \in \mathrm{T}_{\mathcal{A}} \fhh , \quad  \langle \mathcal{X}, \mathcal{Y} \rangle_{\mathcal{A}}  = \sum_{j=1}^{r_1+r_2} \tr(A_j^{-1} X_j \, A_j^{-1} Y_j ), \ee
where $\mathcal{A} = (A_j)_{1 \le j \le r}$, $\mathcal{X} = (X_j)_{1 \le j \le r}$ and $\mathcal{Y} = (Y_j)_{1 \le j \le r}$ are elements from the tangent space $\ds \mathrm{T}_{\mathcal{A}} \fhh = \{ \mathcal{X} = (X_j)_{1 \le j \le r} ; \tr(A_j^{-1} X ) = 0 \}$. Remember that in the sequel the letter $\mathcal{I}$ designates $\mathcal{I}= (\I_n)_{1 \le j \le r} \in \fhh$.

\begin{defi} The \emph{length functions} defined on  $\fhh$ will be the functions $\ell_{V}^{\iota}$, indexed by $\varsigma = (\iota,V)$ where $\iota$ is an integer $1 \le \iota \le h$ and $V$ a flag vectors belonging to $ \Sc_{\sharp}^{\la}(L_{\iota})$ , and defined by :
\be \forall \mathcal{A} \in \fhh, \qquad \ell_{V}^{\iota} (\mathcal{A}) = \ln \left( \frac{ \mathcal{A}[V] }{ \norm(\idA_V^{\iota})^2 \norm ( \ida_{\iota})^{\frac{m}{n}} }  \right) \ee
\end{defi} 
\begin{rema} Actually, in this way, many distinct indexes can parametrise the same length function. For instance, if $\al$ belongs to the field $k$, then the functions  $\ell_{V}^{\iota}$ and $\ell_{\al V}^{\iota} $ coincide.
\end{rema}

Let $V$ be a flag vector of $\Sc_{\sharp}^{\la}(L_{\iota})$ given in the shape :
\newdimen\hoogte    \hoogte=19pt
$$ \hoogte=19pt  
 \breedte=18pt V = \begin{young} x_t \cr \vdots \cr x_1  & \dots & x_1  \cr \end{young} .$$
\newdimen\hoogte    \hoogte=15pt
For any integer $p \le t$, let us agree on $X_{p}$ being the matrix $n \times p$ of which the columns are the vectors $(x_l)_{1 \le l \le p} \in k^n$. The \emph{gradient} expressed at the point $\mathcal{A}$  of  $\ell_{V}^{\iota}$ is the result of the following computation :
\Bea \na \Big( \ell_{V}^{\iota} (\mathcal{A}) \Big) &=& \Big[ d_j \na \big( \ln A_j[V] \big) \Big]_{1 \le j \le r_1 + r_2} \\
&=& \left[ d_j \sum_{l=1}^{s} \na \ln \left( \det ( \transp{{X_{\la_l^*}^{\si_j}}} A_j X_{\la_l^*} ) \right) \right]_{1 \le j \le r_1 + r_2} \\
&=& \left[ d_j \sum_{l=1}^{s}  \left( A_j X_{\la_l^*}^{\si_j} \left( \transp{{X_{\la_l^*}^{\si_j}}} A_j X_{\la_l^*}^{\si_j}  \right)^{-1} \transp{{X_{\la_l^*}^{\si_j}}}A_j - \frac{\la_l^*}{n} A_j \right) \right]_{1 \le j \le r_1 + r_2}
\Eea

Let us the endomorphism $p_{A,X}$ and simply $p_X$ when $A = \Id$ by $p_{A,X} = X (\transp{X}AX)^{-1} \transp{X} A$. This endomorphism turns out to be the $A$-orthogonal projection on the space spanned by the column vectors of $X$. 

\begin{lemm} The gradient of the length functions can be also expressed in the following shape :
\be \na \Big( \ell_{V}^{\iota} (\mathcal{A}) \Big) = \left[ d_j A_j \left(\sum_{l=1}^s p_{A_j, X_{\la_l^*}^{\si_j}} - \frac{m}{n} Id \right)  \right]_{1 \le j \le r_1 + r_2} \ee
\end{lemm}

In particular, the norm of the gradients can be easily computed
\Bea \left\| \na \Big( \ell_{V}^{\iota} (\mathcal{A}) \Big) \right\| &=& 
\sum_{j=1}^{r_1+r_2} d_j^2 \tr \left( \sum_{l=1}^s p_{A_j, X_{\la_l^*}^{\si_j}}  - \frac{m}{n} Id \right)^2 \\
&=& (r_1+4r_2) \left( \sum_{l=1}^s (1 + 2(l-1) - \frac{m}{n}) \la_l^* + \left(\frac{m}{n}\right)^2 \right) \Eea
The norm of the gradients is constant, independant of the length function.

\begin{defi} When $\mathcal{A} \in \fhh$ is fixed,  $S(\mathcal{A})$ shall denote the set of parameters $\varsigma = (\iota, V)$ such that $\ell_V^{\iota} (\mathcal{A})$ is minimal. 
\end{defi}

\begin{lemm} A Hermite--Humbert form  $\mathcal{A}$ being given, there exists only finitely many distinct length functions that reach the minimum $m_{L_{\iota}}(\mathcal{A})$.
\end{lemm}

Indeed, we know already that, decomposing  $\mathcal{A}$ into $\mathcal{A} = (\transp{g} g)$, it is possible to represent $\frac{ \mathcal{A}[V] }{ \norm(\idA_V^{\iota})^2 \norm ( \ida_{\iota})^{\frac{m}{n}} }$ by $H(g \la_{\iota} V)$. The map $X \mapsto H(g \la_{\iota} X)$ defines again a height. Now it is classical that there are only finitely many points of bounded height ; this property is usually known as Northcott property and was initially the point of designing heights.

This finiteness result, associated with the computation of the norm of the gradient, ensures the local finiteness of the set of length functions required to define the Hermite invariant (see remarque 1.1 of \cite{bav2}). This observation is essential to use Ch. Bavard's theory.

According to \cite{bav1}, we set the following definitions :
\begin{defi} A Hermite--Humbert form $\mathcal{A} \in \fhh$  is said to be \emph{perfect} if the familly of gradients $\ds \big( \na \ell_V^{\iota} \big)_{\varsigma = (\iota,V) \in S(\mathcal{A})}$ spans affinely the tangent space $\mathrm{T}_{\mathcal{A}} \fhh$.\\
A Hermite--Humbert form $\mathcal{A} \in \fhh$ is called \emph{eutactic} if the zero vector belongs to the affine interior of the convex span of the family of the gradients $\ds \left( \na \ell_V^{\iota} \right)_{\varsigma = (\iota,V) \in S(\mathcal{A})}$ .
\end{defi}

\begin{rema} If we denote by $\Pi_{\mathcal{A},X}$ the sum of the projections  $\left[ \sum_{l=1}^s p_{A_j, X_{\la_l^*}^{\si_j}} \right]_{1 \le j\le r_1+r_2}$, perfection and eutaxy can be rephrased as follows :
\begin{itemize}
\item A Hermite--Humbert form is called \emph{perfect} if the rank (on $\R$) of the family  $(\Pi_{\mathcal{A},V})_{V \in S(\mathcal{A})}$ is equal to the dimension of the tangent space $T \fhh$ augmented by one :
$$\dim_{\R} \mathrm{T} \fhh + 1 = \frac{r_1 n(n+1)}{2} + r_2n^2 - (r_1+r_2) + 1 .$$
\item A Hermite--Humbert form $\mathcal{A}=\transp{g}g$ is \emph{eutactic} if the identity map $\mathcal{I}$ is a linear combination with only strictly positive coefficients of the  sum of projection maps $(\Pi_{\mathcal{I},gU})_{U \in S(\mathcal{A})}$.
\end{itemize}
\end{rema}

\begin{proof} The rephrasing of the perfection is rooted in the following fact from linear algebra. Let  $E$ be an $\R$ vector space, $H \subset E$ an hyperplane and $u$ an additional vector supplementary to $H$, then the familly  $h_i$ spans affinely $H$ if and only if the familly  $h_i+u$ spans $E$ as a vector space.\\
For the eutaxy, by definition, $\mathcal{A}$ is eutactic if there are positive coefficients  $(\rho_{\varsigma})_{\varsigma \in S(\mathcal{A})}$ of sum equal to $1$ such that
$$0 = \sum_{\varsigma = (U, \iota) \in S(\mathcal{A})} \rho_\varsigma \na \ell_U^{\iota} $$
which is equivalent to, 
$$\Bigg[ \sum_{\varsigma \in S(\mathcal{A})} \rho_\varsigma  \, d_j \, \left(\sum_{l=1}^s A_j X_{\la_l^*} \left( \transp{X_{\la_l^*}} A_j X_{\la_l^*}  \right)^{-1} \transp{X_{\la_l^*}}A_j   \right) \Bigg]_{1 \le j \le r_1+r_2} = \Bigg[\frac{m}{n} A_j \Bigg]_{1 \le j \le r_1+r_2}$$
and using the decomposition $A_j = \transp{g_j} g_j$, we get 
$$\Bigg[ \sum_{\varsigma \in S(\mathcal{A})} \rho_\varsigma  \, d_j \, \transp{g_i} \Pi_{\Id, g_i X} g_i \Bigg]_{1 \le j \le r_1+r_2} = \Bigg[\frac{m}{n} \transp{g_i}g_i \Bigg]_{1 \le j \le r_1+r_2}$$
Whence, 
$$\Bigg[ \sum_{\varsigma \in S(\mathcal{A})} \rho_U  \, \frac{n}{m} \, d_j \, \Pi_{\Id, g_i X} \Bigg]_{1 \le j \le r_1+r_2} = \Bigg[ \I_n \Bigg]_{1 \le j \le r_1+r_2}. $$
\end{proof}

\begin{rema} The notions of perfection and eutaxy coincide with the already defined notions (for example the ones in \cite{cou3} or \cite{cou1}).
\end{rema}

\subsection{A theorem à la Voronoi}

\begin{theo} A Hermite--Humbert form is extreme (with respect to $\la$) if and only if it is perfect and eutactic.
\end{theo}

\begin{proof} It suffices to show that the condition $(\mathscr{C})$ is fulfilled.
 
The proof given here is a staight forward adaptation of the one that can be found paragraph 2.11 of  Ch. Bavard's \cite{bav2}. 
First we go back to the neighbourhood of the identity with the use of the transitive and isometric action $\Phi_{\mathcal{R}}$ of $ \big( \SL_n(\R) \big)^{r_1} \times \big( \SL_n(\C) \big)^{r_2} $ on $\fhh$ which acts by 
$$ \Phi_{\mathcal{R}}  \Big(\mathcal{A}\Big) = \Big( \transp{R_j} A_j R_j \Big)_{1 \le j \le r},  \quad \mathcal{R} \in \big( \SL_n(\R) \big)^{r_1} \times \big( \SL_n(\C) \big)^{r_2}, \; \mathcal{A} \in \fhh$$
where $\mathcal{R} = (R_j)_{1 \le j \le r_1+r_2}$ and  $\mathcal{A} = (A_j)_{1 \le j \le r_1+r_2}$.

Let  $T$ be a finite subset of  $S(\mathcal{I})$, the representatives of which we select in the shape of matrices  $n \times t$ of rank $t$ and the successive columns of which describe the flags (remember that $t = \la_1^*$ refers to the height of the Ferrer diagram of $\la$). Let $\mathcal{X}$ and $\mathcal{Y}$ be two vectors from the tangent space  $\mathrm{T}_{\mathcal{I}} \fhh$, such that $\mathcal{X}$ satisfies the conditions of orthogonality with the elements of $T$.
  
Let us consider the curve  
$$c(t) = \exp( t\mathcal{X} + t^2 \mathcal{Y}^2/2 ) $$
and fix  $f_V^{\iota} = \ell_V^{\iota} \circ c$ ($V \in T$ and $1 \le \iota \le h$). To prove that the condition $\mathscr{C}$ holds, we need to expand$f_V^{\iota}(t)$ up to the fourth order. We have
$${f_V^{\iota}} ' (0) = 0 ,$$
because of the orthogonality conditions on $\mathcal{X}$, and
$${f_V^{\iota}} '' (0) = \sum_{j=1}^{r_1+r_2} d_j \left( \sum_{l=1}^s  \tr \big(p_{V[\la_l^*]^{\si_j} } Y_j\big) + \tr \big(p_{V[\la_l^*]^{\si_j} } X_j^2\big)  - \tr \big( \, (p_{V[\la_l^*]^{\si_j} } X_j)^2 \, \big)    \right)$$

Remember that the notation $p_V$ refers to the projection matrix  $V( \transp{V}V)^{-1} \transp{V}$ and that $V[\ell]$ is the matrix $n \times \ell $ got by extracting the $\ell$ first columns of $V$.

One can check that for any symmetric matrix $Z$, 
$$ \tr \big(( p_V Z )^2\big) \le \tr\big(p_V Z^2\big) $$
with equality if and only if $Z$ commutes with $p_V$. Let us define thus the subset $T_0 \subset T $ of the parametres $U \in T$ such that for any $j \in \Np{1}{r_1+r_2}$ and any $l \in \Np{1}{s}$,
$$p_{U[\la_l^*]^{\si_j} } X_j =  X_j p_{U[\la_l^*]^{\si_j} } $$
as well as $T_1 = T \smallsetminus T_0$ its complement.
For $U$ in $T_0$, one has also
$${f_U^{\iota}}^{(3)} (0) = 0 $$
and
$${f_U^{\iota}}^{(4)} (0) = \sum_{j=1}^{r_1+r_2} d_j \left( \sum_{l=1}^s \tr \big(p_{U[\la_l^*]^{\si_j} } Y_j^2\big) - \tr \big( \,(p_{U[\la_l^*]^{\si_j} } Y_j)^2 \,\big)\right)$$
These computation can be drawn from the computation of the expansion of  $\det(\transp{U} A U)$ by Ch. Bavard \cite{bav1} p 111 and from the following identity 
$$\ln \left( 1 + \ba \frac{t^2}{2} + \de \frac{t^4}{24} \right)= \ba \frac{t^2}{2} + (\de - 3 \ba^2) \frac{t^4}{24} + o (t^4) .$$

We are looking for an $r$-upple of matrices $\mathcal{Y} = (Y_j)_{1 \le j \le r}$ such that for $U \in T_1$, ${f_U^{\iota}} '' (0)>0$ and for  $U \in T_0$, ${f_U^{\iota}} '' (0)=0$ and ${f_U^{\iota}}^{(4)} (0)>0$. The first conditions handling  $U \in T_1$ can always be satisfied provided  $\mathcal{Y}$ is replaced by $\eps \mathcal{Y}$ with $\eps >0$ small enough. The second conditions handling $U \in T_0$ are equivallent to $\sum_{j=1}^{r} \tr \big(p_{V[\la_l^*]^{\si_j} } Y_j\big) = 0$ and there exists a pair $(j_0, \ell_0)$ such that $p_{U[\la_l^*]^{\si_j} } $ and $Y_j$ do not commute. The same argument used to build $Y$ in the proof of proposition 2.8 of \cite{bav1} applies here, for any pair of indices $(j_0,\ell_0)$, which supplies us with a matrix $Y_{j_0}$. For the other indices $j \neq j_0$, $Y_j$ can be chosen equal to $0$ for instance.
\end{proof}

\subsection{Algebraicity of the constant}

\begin{prop} \begin{enumerate} \item For a given integer $n$, partition $\la$ and number field $k$, there are only finitely many perfect forms up to unimodular transformations.
\item The Hermite constant $\ga_{n, \la}$ is algebraic.
\end{enumerate}
\end{prop}

\begin{proof}  2. This property is a quite general fact that have also been mentioned in \cite{bav2}. The proof goes as follows. Perfection for a Hermite--Humbert form $\mathcal{A}$ means that the algebraic subvariety $\mathcal{C}_{S(\mathcal{A})}$ defined by the polynomial equations $\ds \ell_U^{\iota} (\mathcal{X}) = 1$ for  $ (\iota,U) \in S(\mathcal{A})$ is of dimension zero. Thus, the equations defining $\mathcal{C}_{S(\mathcal{A})}$ being all polynomial with rational coefficients, the points in $\mathcal{C}_{S(\mathcal{A})}$ are algebraic and so is in particular the form $\mathcal{A}$. 

1. Now to show that there are only finitely many perfect forms up to unimodular transformations, we show that there is always a representative of a perfect form that takes its minimal flag vectors among a finite set. Thus there can be only finitely many set of equation $\ds \ell_U^{\iota} (\mathcal{X}) = 1$ for  $ (\iota,U) $ belonging to some $ S_0$ defining a class of perfect forms.

By Humbert reduction theory \cite{hum}, a Humbert form $\mathcal{A}$ can always be expressed up to unimodular equivalence as $\mathcal{A} = (D_j[U_j]])_{1 \le j \le r}$ where the $(D_j)$ are diagonal matrices such that the diagonal coefficients $d_j(i)$ satisfy $\frac{d_j(i)}{d_j(i+1)} \le B$ for some positive bound $B$ and such that the $(U_j)$ are unipotent upper triangular matrices with bounded coefficients.

We recall that the local height of a flag vector $X \in \Sc_{\sharp}^\la$ is just the square norm $A_j[X] = \| \pi_\la (D_j U_j) X \|^2$ where $\| \cdot \|$ is the norm we defined on $\Sc^{\la}(k_{v_j})$. Since $\pi(D_i U_i)$ is invertible, we get, using the operator norm, 
$$\forall X \in \Sc^\la (k_v), \quad  A_j[X]  \ge \| \pi_{\la}(D_jU_j)^{-1} \|^{-1} \; \| X \|^2 $$
Since, $U_j$ is triangular and unipotent, the entries of $\pi_\la(U_j^{-1})$ are polynomial in the entries of U, and $\| \pi_\la(U_j^{-1}) \|$ can be uniformely bounded for the reduced Humbert forms we consider. On the other hand $\pi_\la(D_j^{-1})$ is a diagonal endomorphism which eigenvalue associated to $e_T$ is just $\prod_{i=1}^n d_j(i)^{-\# \{i ; i \in T\} }$. Using the bounds on the ratios of to consecutive $d_j(i)$, we notice that the eigenvalues of $\pi_\la(D_j^{-1})$ are bounded from above by  $B^{\mu } \prod_{i=1}^n d_j(i)^{-\la_i }$ for some big power $B^\mu$. Thus, there is a constant $l$ such that for any Humbert form $\mathcal{A}$, 
$$\forall X \in \Sc^\la (k_v), \quad  A_j[X]  \ge l \;  \| X \|^2 $$

We deduce that the minimal flag vectors $X$ of Humbert form have a bounded height : their infinite part is for instance bounded by $\left( \frac{1}{l} \right)^r$ whereas their finite part is always less than one. According to Northcott property, there can only be finitely many of them, which ends the proof.
\end{proof}

\section{Some relations between the constants}
\subsection{An equality of duality}

If $\la$  is a partition with less than $n$ parts, we call complementary partition with respect to $n$ the partition $\overline{\la}$ (also denoted by $\overline{\la}^n$) such that for any  $\ell$ between  1 and $s$, $\la^*_{\ell} + \overline{\la}^*_{s+1-\ell} = n$. Visually, it can be retrieved by completing the partition into a rectangle of height $n$ :

$$ \hoogte=15pt  
 \breedte=18pt \begin{young} \; \times &\;\times &\; \times \cr \; \times & \; \times & \; \times \cr \,\bigcirc & \; \times & \; \times \cr \; \bigcirc & \, \bigcirc & \; \times \cr \, \bigcirc & \, \bigcirc & \, \bigcirc  \cr \end{young}
\quad \begin{matrix} \uparrow \\ n \\ \downarrow \end{matrix}  $$

\begin{prop} Let $\la$ be a partition and $\overline{\la}^n$ complementary partition with respect to $n$, then the following equality holds
\be \ga_{n, \la} = \ga_{n, \overline{\la}^n} \ee
\end{prop}

\begin{proof} A partition $\la$ being fixed, we consider the following representation $(\rho, \Sc^{\la}(k^n)$ where $\rho = \pi_{\la}(w_0 \transp{g}^{-1} w_0^{-1})$ and $w_0$ is the miror automorphism of $\GL_n(k)$ which swaps the vectors $e_i$ and $e_{n+1-i}$. We can notice that $\rho$ is an irreducible representation, that $e_{U(\la)}$ directs the line of highest weight vectors, that $e_{U(\la)}$ is stabilised by the parabolic subgroup $P_{\overline{\la}}$. The character of $\rho$ is actually $\det^{-n} \cdot \chi_{\overline{\la}}$.

Taking our normalisation into account, {\it i.e.} $H(e_{U(\la)}) = H(e_{U(\overline{\la})}) = 1$, when $g$ belongs to $ \GL_n^1(\A_k)$ and decomposes into $g = k d u$ with $k$ in  the maximal compact subgroup $K(\A_k)$, $d$ diagonal matrix and $u$ unipotent, we get 
$$H(\rho(g) e_{U(\la)})= |\chi_{\overline{\la}}(d)|_{\A_k} = H(\pi_{\overline{\la}}(g) e_{U(\overline{\la})}) $$

We derive the equality of the constants $\ga_{n, \la} = \ga_{n, \overline{\la}^n}$ from this equality.
\end{proof}

\subsection{Mordell inequality}

With the view point of twisted heigths in mind, we can show the following inequality, which generalises Mordell inequality.

\begin{prop} Let $\la$ be a partition, $m$ and $n$ two integers such that $t \le m \le n$, then, 
\be \ga_{n, \la} \le \ga_{m, \la} \left( \ga_{n,m} \right)^{|\la|/m} \ee
\end{prop}

\begin{proof} Consider an automorphism $A$ satisfying $|\det(A)|_{\A_k} = 1$. Let $\mathscr{D}$ be a flag of $k^n$ which minimises the height $H_A(\mathscr{D})$ and let $W$ be a subspace of dimension  $m$ such that $H_A(W) \le \ga_{m,n}^{1/2}$. There exists an injective map $\phi$ which sends $k^m$ onto $W \subset k^n$. Let us accept for one time the following lemma, which will be proven later. (This lemma and its proof easily stems from corollary 4.3 in \cite{rt} where this result is proven for symmetric powers and exterior powers.)

\begin{lemm} Let $\phi : k^m \hookrightarrow k^n$ be an injective map and $A$ an automorphism of $\GL_n(\A_k)$. We denote also by $\phi$ the $\A_k$-homomorphism which extends  $\phi$ from $\A_k^m $ to $\A_k^n$. There exists an automorphism $B \in \GL_m(\A_k)$ such that the twisted height $H_B$ coincides with $H_A$ in the following sense : for any partition $\la$ and any flag $\mathscr{D}$ of shape  $\la$ of nested subspaces in $k^m$, we have
\be H_A(\phi( \mathscr{D})) = H_B(  \mathscr{D} ) .\ee
In particular, when $\la$ is the partition $\la=
 \hoogte=2pt 
 \breedte=3pt
 \begin{young}
 \cr \cr \cr
\end{young} 
 \hoogte=15pt  
 \breedte=18pt 
  $ (with $m$ vertically arranged boxes) and if $W$ means the image of  $\phi$, we get 
\be H_A(W) = |\det(B)|_{\A_k}. \ee
\end{lemm}
Let us fix an automorphism $B \in \GL_m(k)$ enjoying the properties of the lemma. 

There exists besides a flag $ \mathscr{T}$ of  $k^m$ such that 
$$ H_B(\mathscr{T}) \le \ga_{m,\la}^{1/2} |\det ( B) |_{\A_k}^{|\la|/m} =  \ga_{m,\la}^{1/2} H_A(W)^{|\la|/m} \le \ga_{m,\la}^{1/2} \ga_{n, m}^{|\la| /2 m} .$$
Then, 
$$H_A(\mathscr{D}) \le H_A(\phi( \mathscr{T} )) = H_B( \mathscr{T} ) \le \ga_{m,\la}^{1/2} \ga_{n, m}^{|\la /2 m} $$
which ends up the proof of the proposition.
\end{proof}

\begin{proof}[Proof of the lemma] Let us start with building automorphisms $B_v \in \GL_n(k_v)$ for any place $v$ such that $A_v \circ \phi \circ B_v^{-1}$ preserves the norm. To that purpose, if $v$ is an archimedean place, we consider the preimages of a family of  $m$  $A_v$-orthonormal vectors  of $k_v^n$ and take as $B_v$, the automorphism which sends the canonical basis of $k_v^m$ on these vectors. When $v$ is an ultrametric place, to have the norm preserved, it is necessary and sufficient that $A_v \circ \phi \circ B_v^{-1}$ send $\ido_v^m$ on a primitive $\ido_v$-module of rank $m$ in  $\ido_v^n$. We choose $B_v$ such that $B_v(\ido_v^n) = (A_v \circ \phi)^{-1} (\ido_v^n)$.

Let us notice that for almost any finite place, $A_v \circ \phi$ is already an isometry, and that we can content ourselves with $B_v = \I_n$. This ensures that $B = (B_v)_{v \in \val}$ is really an element of $\GL_m(\A_k)$ and thus
$$ \forall x \in k^m, \quad H_A(\phi(x)) = H_B(x) $$

Let us show that this equality extends to flags of any shape. We can decompose any map $A_v \circ \phi \circ B_v^{-1}$ into a composition $\psi_v \circ \iota$ where $\iota$ is the injection $k^m \hookrightarrow k^n$ given by $(x_1, \dots, x_m) \mapsto (x_1, \dots, x_m, 0, \dots, 0)$ and $\psi_v$ is an isometry. Then the map $\Sc^{\la}(A_v \circ \phi \circ B_v^{-1}) = \pi_{\la} (\psi_v) \circ \Sc^{\la}(\iota)$ is an isometric injection of $k_v^m$ into $k_v^n$ since on the one hand, $\psi_v$ is an isometry and by construction our local heights are invariant under the action of an isometry, and on the other hand, the injection $\Sc^{\la}(\iota)$ is an isometry. Thus for any partition $\la$ and any flag $\mathscr{D}$ of shape $\la$, 
$$ H_A(\phi( \mathscr{D})) = H_B(  \mathscr{D} ) .$$
\end{proof}

\subsection{An inequality involving Berg\'e--Martinet constant}

\begin{defi} Let us recall that for a lattice the Berg\'e--Martinet constant means the maximum of the product of the minimum of a lattice by the minimum of the dual lattice. In adelic terms, it can be expressed like this
$$ \ga_{n,\la}' = \max_{g \in \GL_n(\A_k)} \left( \min_{\ga \in \GL_n(k)}  H(\pi_{\la}(g\ga) e_{U(\la)} )  \min_{\ga \in \GL_n(k)}  H(\pi_{\la}( \transp{g}^{-1} \ga ) e_{U(\la)} ) \right)^{1/2}. $$
\end{defi}

Of course, it is true that for any partition $\la$, the inequality $\ga_{n,\la}' \le \ga_{n,\la} $ holds. Besides

\begin{prop} \label{p:ibm} Let $\kappa$ be the partition $\kappa = (n-1,1) =  \hoogte=2pt 
 \breedte=3pt
 \begin{young}
  \cr 
  \cr
  \cr
 &  \cr
\end{young} 
 \hoogte=15pt  
 \breedte=18pt  $, then $\ga_{n,(1)}'^2 \le \ga_{n,\kappa} $.
\end{prop}

\begin{proof}
Let us denote $g^c = w_0 {}^{t}g^{-1} w_0^{-1}$ where $w_0$ is the miror automorphism which swaps the vector $e_i$ with $e_{n+1-i}$. We can notice that for any diagonal automorphism $g = \text{Diag}(d_1, \dots, d_n)$, we have $H(ge_1)H(g^c (e_1)) = F(d_1 e_1) H(d_n^{-1} e_1) = \| d_1 \|_{\A_k} \| d_1 \dots d_{n-1} \|_{\A_k}$ since $\|\det g\|_{A_k} =1$. Thus, we have the equality $H(g e_1) H(g^c e_1) = H(\pi_{\kappa} (g) e_{U(\kappa)})$ for any diagonal matrix. The equality holds also trivially for unipotent upper triangular matrices (all the terms are equal to one) and extends to any element $g \in \GL(\A_k)$. 

Now we have for any $g \in \GL_n(\A_k)$, 
$$ \min_{\ga \in \GL_n(k)}  H( g\ga e_1 )  \min_{\ga \in \GL_n(k)}  H(\transp{g}^{-1} \ga e_1 ) \le \min_{\ga \in \GL_n(k)}  H(g\ga e_1 )    H( g^c \ga  e_1 )$$
Thus
$$\min_{\ga \in \GL_n(k)}  H( g\ga e_1 )  \min_{\ga \in \GL_n(k)}  H(\transp{g}^{-1} \ga e_1 )= \min_{\ga\in \GL_n(k)}  H(\pi_{\kappa}(g\ga) e_{U(\kappa)} ) $$
which leads to the equality we wanted to prove.
\end{proof}

\section{Some exact values and upper bounds}

\newdimen\hoogte    \hoogte=2pt 
\newdimen\breedte   \breedte=3pt
\subsection{Determination of $\ga_{3,  (2,1) }(\Q)$ and of $\ga_{4, (3,1)}(\Q)$}
\newdimen\hoogte    \hoogte=15pt  
\newdimen\breedte   \breedte=18pt

\begin{prop}
The constant $\ga_{3, \newdimen\hoogte    \hoogte=2pt 
\newdimen\breedte   \breedte=3pt
 \begin{young}
 \cr
 &  \cr
\end{young}
\newdimen\hoogte    \hoogte=15pt  
\newdimen\breedte   \breedte=18pt}(\Q)$ is equal to $\frac{3}{2}$ and is achieved only for the root lattice $\A_3$ and its dual $\A_3^*$.  

The constant $\ga_{4, \newdimen\hoogte    \hoogte=2pt 
\newdimen\breedte   \breedte=3pt
 \begin{young}
 \cr
 \cr
 &  \cr
\end{young}
\newdimen\hoogte    \hoogte=15pt  
\newdimen\breedte   \breedte=18pt}(\Q)$ is equal to $2$ and is achieved only for the root lattice  $\D_4$ (which is isomorphic to its dual).  
\end{prop}

\begin{proof} For any reference to the reduction in the sense of Korkine and Zolotareff, we send back the reader to \cite{mar} section 2.9 or to the original article \cite{kz}. Let $\La$ be a lattice of $\R^3$ and $L$ a sublattice of $\La$, the $2\times 2$  determinant of which is  minimal. Two cases can occur depending on whether the sublattice $L$ can be found containing a minimal vector or not. 
\begin{enumerate} \item In the first case, let $u_1$ be a minimal vector enjoying such properties and let $u_2$ be a second vector of $L$ such that $(u_1, u_2)$ forms a basis of $L$. Now, $u_1$, $u_2$ is the beginning a Hermite--Korkine--Zolotareff reduced basis of $\La$, say $(u_1, \,u_2, \, u_3)$. Were it not the case, the two first vectors of an other HKZ reduced basis would provide us a better sublattice $L$. In such event, denoting $A_1$ the norm of $u_1$, $A_2$ the norm of the projection of $u_2$ on the orthogonal of $u_1$ and $A_3$ the norm on the projection of $u_3$ on the orthogonal of $u_1$ and $u_2$, the constant is 
$$\ga_{3, \, \newdimen\hoogte    \hoogte=2pt 
\newdimen\breedte   \breedte=3pt
 \begin{young}
 \cr
 &  \cr
\end{young}
\newdimen\hoogte    \hoogte=15pt  
\newdimen\breedte   \breedte=18pt}(\La) = \frac{A_1^2A_2}{A_1A_2A_3} = \frac{A_1}{ A_3}  $$
It has been disclosed by Korkine and Zolotareff that this ratio never exceeds $\frac{3}{2}$ and can only be reached when $\La=\A_3$ or $\La=\A_3^*$.  
\item The second case to study corresponds to the situation where no dimension $2$ sublattice $L$ with minimal determinant bears a minimal vector. Consider a reduced basis  $(u_1,u_2)$ of the lattice and $e_m$ a minimal vector of $\La$. Then the triple $(u_1,u_2,e_m)$ forms a basis of $\La$. Indeed, assume that there exists an other vector $x$ of $\La$ which is not contained in the lattice spanned by this triple. Even when it means performing some reductions, one can assume the  $\langle x, e_m \rangle \le \frac{1}{2} \|e_m|$, $\langle x, u_1 \rangle \le \frac{1}{2} \|u_1|$ and $\langle x, u_2 \rangle \le \frac{1}{2} \|u_2|$. Then the determinant of the lattice  $L' = \Z x + \Z u_1$ is bounded from above by  
$$\det L \le \|x\|^2  \|u_1\|^2 \le ( \frac{1}{2}\|e_m\|^2 +  \frac{1}{2}\|u_1\|^2 +  \frac{1}{2}\|u_2\|^2 )  \|u_1\|^2 <  \frac{3}{4}\|u_2\|^2 \|u_1\|^2$$
But the properties of reduction of the basis  $u_1,u_2$ imply that $\det L \ge \frac{3}{4}\|u_2\|^2 \|u_1\|^2$.\\
Let us denote by $A_3$ the norm of the projection of  $e_m$ on the orthogonal of $L$. We dispose of the chain of inequality $\|e_m\|^2 \le \|u_1\|^2 = A_1 \le \frac{4}{3} A_2$ since $e_m$ is a minimal vector and  $(u_1,u_2)$ is a reduced base. Besides, comparing the determinants of the lattices $L$ and $\Z e_m + \Z u_1$, there arises $A_1 A_3 \ge A_1 A_2$. Thus $\|e_m\|^2 \le  \frac{4}{3} A_3$. As a result, in this second case,
$$\ga_{3, \, \newdimen\hoogte    \hoogte=2pt 
\newdimen\breedte   \breedte=3pt
 \begin{young}
 \cr
 &  \cr
\end{young}
\newdimen\hoogte    \hoogte=15pt  
\newdimen\breedte   \breedte=18pt}(\La) = \frac{A_1 A_2 \|e_m\|}{A_1A_2A_3} \le \frac{4}{ 3}  $$
which is a lowerer bound than in the first case.
\end{enumerate}

{\it Mutatis mutandis}, if $\La$ is a dimension 4 lattice, two cases are to be distinguished, whether there exists or not a dimension 3 sublattice $L$ of minimal determinant which contains a minimal vector. In the first case, a HKZ reduced basis can be exhibited wherein the constant can be expressed as 
$$\ga_{4, \,  \newdimen\hoogte    \hoogte=2pt 
\newdimen\breedte   \breedte=3pt
 \begin{young}
 \cr
 \cr
 &  \cr
\end{young}
\newdimen\hoogte    \hoogte=15pt  
\newdimen\breedte   \breedte=18pt}(\La) = \frac{A_1^2A_2 A_3}{A_1A_2A_3 A_4} = \frac{A_1}{ A_ 4}  $$
and is bounded from above by $2$ according to the work of Korkine and  Zolotareff. This upper bound can only be reached when $\La = \D_4$.\\
In the second case, a basis of $\La$ can be built by appending a minimal vector to a reduced basis of a minimal dimension 3 sublattice $L$. It can be shown that  $\ga_{4, \newdimen\hoogte    \hoogte=2pt 
\newdimen\breedte   \breedte=3pt
 \begin{young}
 \cr
 \cr
 &  \cr
\end{young}
\newdimen\hoogte    \hoogte=15pt  
\newdimen\breedte   \breedte=18pt}(\La) \le  \frac{3}{ 2}  $ in that case.
\end{proof}

\begin{rema} It appears from these determinations that  \newdimen\hoogte    \hoogte=2pt 
\newdimen\breedte   \breedte=3pt
$\ga_{3,  \begin{young}
 \cr
 &  \cr
\end{young}}(\Q)$ et $\ga_{4,  \begin{young}
 \cr
 \cr
 &  \cr
\end{young}}(\Q)$
\newdimen\hoogte    \hoogte=15pt  
\newdimen\breedte   \breedte=18pt
are exactly equal to $\ga_3'(\Q)$ and $\ga_4'(\Q)$ respectively (equality case in the proposition \ref{p:ibm}). \\
In dimension 5, the inequality becomes strict. According to  \cite{bm}, let us call $\ga_5''$ the upper bound to the quantity $\frac{A_1}{A_5}$ which appears in the HKZ reduction of a form with more than five variables. The value of $\ga_5''$ is not known but we dispose of the bounds $\frac{32}{5} \le \ga_5'' < \frac{9}{4}$. They enable us to prove with the same arguments as above that
\hoogte=2pt 
\newdimen\breedte   \breedte=3pt
$\ga_{5,  \begin{young}
 \cr
 \cr
 \cr
 &  \cr
\end{young}}(\Q) = \ga_5'' $
\newdimen\hoogte    \hoogte=15pt  
\newdimen\breedte   \breedte=18pt
whereas the value of $\ga_5'$ is 2, according to the computations of \cite{py}.
\end{rema}

\subsection{Upper bound through the second Minkowski theorem}

Let $A$ be an automorphism of  $\GL_n(\A_k)$ and  $\mathscr{D}$ a flag of the shape $\la$ and $(x_i)_{1 \le i \le t}$ a sequence of vectors that spans $\mathscr{D}$, then the following Hadamard like inequality is checked : 
\be \label{e:61} H_A(\mathscr{D}) \le \prod_{\ell=1}^t H_A(x_\ell)^{\la_{\ell}^*}. \ee
Indeed, up to a transformation of $A$, it suffices to ensure that this inequality holds when the flag $\mathscr{D}$ is the flag built starting with the canonical basis, that is when $x_{\ell} = e_{\ell}$ for any $\ell$. The automorphism $A$ can be decomposed into  $A = k d u$ where $k$ belongs to  $K_n(\A_k)$, $d$ is a diagonal matrix, the coefficients of which are, say, $d_i \in \A_k^{\times}$ and $u$ is an unipotent upper triangular matrix. The action of $k$ does not modify the values of the terms that appear on the two sides of the inequality (\ref{e:61}). The action of the product $d u$ on the vector $e_{U(\la)}$ boils down to multiplying the height by the quantity  $|\chi_{\la}(d)|_{\A_k}$ on the left hand side ; whereas for the right hand side, $H(ue_i) \ge H(e_i)$ et $H(due_i) = |d_i|_{\A_k} H(ue_i) \ge d_i H(e_i)$, which ends up the proof of (\ref{e:61}). 

According to the adelic version of the second Minkowski theorem for convex bodies, (see \cite{mcf} or \cite{thu3}), for a fixed automorphism $A$, there exists a basis of  $k^n$ such that 
$$\prod_{\ell=1}^n H_A(x_\ell) \le \frac{2^{nr}D_k^{n/2} }{V(n)^{r_1} V(2n)^{r_2}} |\det A|_{\A_k} $$
where $D_k$ is the discriminant of $k$ and $V(k)$ the volume of the unit ball of dimension $n$.
We can assume without loss of generality that $H_A(x)\le H_A(x_2) \le \dots \le H_A(x_n)$, which allows us to write
$$\left( \prod_{\ell=1}^t H(x_\ell)^{\la_{\ell}^*} \right)^n \le \left( \prod_{\ell=1}^n H_A(x_\ell)  \right)^{|\la|} $$
and to conclude that  
\begin{prop} The following inequality holds :
\be \ga_{n,\la}(k)^{1/ 2|\la|} \le  \frac{2^{r}D_k^{1/2} }{V(n)^{r_1/n} V(2n)^{r_2/n}}  \ee
where $D_k$ is the discriminant of $k$ and $V(k)$ the volume of the unit ball of dimension $n$.
\end{prop}

\subsection{Upper bound by changing the base field}

The following lemma can be proven by simultaneous diagonalisation.
\begin{lemm} The map $\psi : \mathscr{H}_n^{++} \to \R$ defined by $\psi = \ln \circ \det$ is concave. In particular, if
$(A_j)_{1 \le j \le p}  \in (\mathscr{H}_n^{++})^p$, then, according to Jensen inequality,
\be \left( \det \left( \prod_{j=1}^{p} A_j \right) \right)^{\frac{1}{p}} \le  \\  \frac{ \det \left( \sum_{j=1}^{p} A_j \right) }{p}  \ee
\end{lemm}

It enables us to exhibit the following inequality, demonstrated in \cite{ohwa} for the case $\la =(1)$.

\begin{theo} If $\la \vdash m $ is a partition of $m$, if $D_k$ designates the discriminant of the field $k$, the following inequality is true
\be \ga_{n, \la} (k) \le \frac{|D_k|^{m} (\ga_{nd, \la}(\Q))^d }{d^d} \ee
\end{theo}

\begin{rema} We need here to give a reference to the size $n$ of the involved group $\GL_n$, which we do by completing the notation $\ga_{n, \la}$, not to be confused with the notation $\ga_{nd, \la}$, relative to the group $\GL_{nd}$. 
\end{rema}

\begin{proof}
The idea of the demonstration is to transform all the  $\ido$-modules into $\Z$-modules. To that end, we introduce like in \cite{ohwa} scalars of $k$ $(u_1^{(\iota)} \dots, u_d^{(\iota)})$ that constitute a $\Z$-basis of the ideal $\ida_{\iota}$. Then we can consider the product basis $\mathcal{B}^{(\iota)}$ of $L_{\iota}$ seen as a $\Z$-module which consists in the vectors $\eps_{j,l}^{(\iota)} = u_j^{(1)} e_l$ for $1 \le j \le d$ and $1 \le l \le n-1$ and the vectors $\eps_{j,n} = u_{j}^{(\iota)} e_n$ for $1 \le j \le d$. Hereupon, we associate to any Hermite--Humbert form $\mathcal{A} \in \fhh$ the quadratic form $\Phi^{(\iota)}$ defined on $\Q^{nd}$ and given by the following formula where $y$ belongs to $\Q^{nd}$ and $Y$ is the vector of $k^n$ of coordinates $y$ in the basis $\mathcal{B}^{(\iota)}$
\be \Phi^{(\iota)} (y) = \sum_{j=1}^{r_1} \transp{Y} A_j Y + 2 \sum_{j=r_1+1}^{r_2} \transp{Y} A_j Y. \ee

For any $t$-upple $(Y_1, \dots, Y_t) \in \left. L_{\iota} \right.^t $, the coordinates of which in $\Q^{nd}$ are $(y_1, \dots, y_t) \in \left( \Q^{nd} \right)^t$, and such that the vector $U$ below is non zero :
\newdimen\hoogte    \hoogte=19pt
\newdimen\breedte   \breedte=18pt
$$ V = 
\begin{young} 
Y_t \cr 
\vdots \cr 
Y_1 & \dots & Y_1 \cr \end{young}
 \in \Sc_{\sharp}^{\la}(L_{\iota}) \qquad v = \begin{young} y_t \cr \vdots \cr y_1 & \dots & y_1 \cr \end{young} \in \Sc_{\sharp}^{\la}(\Q^{nd})$$
\newdimen\hoogte    \hoogte=15pt
we have the inequality
\be \left( \mathcal{A}[V] \right)^{\frac{1}{d}} \le \frac{\Phi^{(\iota) }[v]}{d} \ee
In particular, passing up to the minimum on the $t$-upples, the definition of  $\ga_{\la}^{nd}(\Q)$ enables us to write
\be \min_V \left( \mathcal{A}[V] \right)^{\frac{1}{d}} \le \frac{ \ga_{\la}^{nd}(\Q) \; \det( \Phi^{(\iota)})^{\frac{m}{nd}} }{d} \ee
The determinant of $\Phi^{(\iota)}$ is detailed in \cite{ohwa}, its value is
$$ \det (\Phi^{(\iota)}) = \norm( \ida_{\iota}) |D_k|^n \det \mathcal{A} = |D_k|^n {\det}_{L_{\iota}} \mathcal{A} $$
Thus
$$ \min_{V \in \Sc_{\sharp}^{\la}(L_{\iota})} \left( \mathcal{A}[V] \right) \le \frac{ (\ga_{\la}^{nd}(\Q))^d \; |D_k|^{m} ({\det}_{L_{\iota}} \mathcal{A})^{\frac{m}{n}} }{d^d} $$
Since the ideal $\idA_Z$ is always integral, we have even
\be \frac{m_{L_{\iota}} (\mathcal{A}) }{ ({\det}_{L_{\iota}} \mathcal{A})^{\frac{m}{n}} } \le \frac{ (\ga_{\la}^{nd}(\Q))^d \; |D_k|^{m}  }{d^d} \ee
Whence we get easily the expected inequality
\end{proof}

\paragraph*{Aknowledgement } I would like to thank Renaud Coulangeon for acquainting me with the subject and for his constant support  as well as Takao Watanabe for welcoming me in Osaka and answering my numerous questions.

\bibliographystyle{alpha}
\bibliography{CstHerGalVor.bib}

\end{document}